\documentclass[12pt]{article}

\usepackage{amsmath,amsfonts,amssymb,amsthm,mathrsfs}
\usepackage{lmodern}
\usepackage{tikz-cd}
\usepackage{hyperref}

\newtheorem{Th}{Theorem}[section]
\newtheorem{Cor}[Th]{Corollary}
\newtheorem{Lem}[Th]{Lemma}

\theoremstyle{definition}
\newtheorem{Def}[Th]{Definition}

\newtheorem{Rem}[Th]{Remark}

\newcommand{\N}{\mathbb N}

\begin{document}

\title{\bf Quotient of  Topological Ternary Semigroup }
\author{S. Samanta.\footnote{samantapmju@gmail.com;
Department of Mathematics, Goenka College of Commerce and B.A.;
210 B. B. Ganguly Street, Kolkata -- 700012, India.} ,
S. Jana.\footnote {sjpm@caluniv.ac.in;
Department of Pure Mathematics, University of Calcutta; 35, Ballygunge Circular Road, Kolkata --- 700019, India.}  \&
S. Kar.\footnote{karsukhendu@yahoo.co.in; Department of
Mathematics, Jadavpur University; 188, Raja S. C. Mallick Road,
Kolkata --- 700032, India.}}

\date{}
\maketitle

\begin{abstract}
In this paper we introduce a quotient structure on topological ternary semigroup by defining a congruence suitably. We have found conditions under which this quotient structure becomes a topological ternary semigroup. We have also obtained conditions that make this quotient a topological ternary group, whenever the base structure is a topological ternary group.
\end{abstract}


\section{Introduction}

A. D. Wallace considered and studied the problem of making quotient of a topological semigroup by means of a suitable congruence. In the paper \cite{Wallace} he proved the following: If $S$ is a compact topological semigroup and $\rho$ is a closed congruence on $S$ then $S/\rho$ becomes a compact topological semigroup. Lawson and Madison studied Rees version of the problem and in \cite{Lawson} they generalized Wallace's result on a locally compact $\sigma$-compact topological semigroup. Later many mathematicians  like Gonzalez \cite{Gonzalez}, Gutik and Pavlyk \cite{Gutik}, Khosravi \cite{Behnam} used algebraic and topological conditions to make the quotient space a topological semigroup.

In the present paper we study quotient structure of a topological ternary semigroup with the help of a congruence. We have found conditions to make the quotient space of a topological ternary semigroup again a topological ternary semigroup with respect to the quotient topology. To resolve this issue we have faced mainly two problems: the first one is to show the continuity of induced ternary multiplication in quotient space and second one is to prove the quotient space a  Hausdorff space, since it is not preserved under continuity. 
After this we use it to generalize Wallace's result in ternary system. Then we prove that the result is valid if  compactness is replaced by $k_{w}$-space. Thereafter we generalize Lawson and Madison's theorem in ternary system. Finally, we study the quotient structure of a topological ternary group. This is done with the help of a normal ternary subgroup. We prove that this quotient space becomes a topological ternary group under some suitable condition.

 \section{Prerequisites }
 In this section we mainly mention some prerequisites, most of which are available in the literature of ternary semigroup theory and  few standard  topological definitions that are required for our main discussion. 
 
 \begin{Def} \cite{Bro}
 A non-empty set $S$ together with a ternary operation, called \emph{ternary multiplication} and denoted by juxtaposition, is said to be a \emph{ternary  semigroup}  if\\ \centerline{$(abc)de=a(bcd)e=ab(cde)$, for all $a, b, c, d, e\in S$.}
 \end{Def}

 \begin{Def}\cite{Km} An element $a$ of a ternary semigroup $S$ is  said to be \emph{invertible} in $S$ if there exists an element $b\in S$ such that $ abx = bax = xab = xba = x$, for all $x\in S$. Then $b$ is called the \emph{inverse} of $a$ (it is unique, if exists).
 \end{Def}
 \begin{Def} \cite{Bro} A ternary semigroup $S$ is called a \emph{ternary group} \index{Ternary group} if for $a,b,c\in S$ the equations $abx=c, axb=c$ and $xab=c$ have solutions in $S$.
 	
 	A non-empty subset $T$ of a ternary semigroup $S$ is called a \emph{ternary subgroup} of $S$ if $T$ is a ternary group with respect to the ternary operation of $S$ restricted to $T$.
 	\label{ch1:d:tg}.\end{Def}

 \begin{Def}\cite{Sj}
 	For any ternary subgroup $N$ of a ternary group $T$, the sets $aNN (NNa)$, $\forall \,\ a\in S$, are called \textit{left(right) cosets} of $N$ in $T$.
 \end{Def}
 
 \begin{Def} \cite{Sj} A ternary subgroup $H$ of a ternary group $S$ is said to be \emph{normal ternary subgroup} of $S$ if it satisfies any one of the following equivalent conditions:\\
 	 (a) $gHg^{-1}\subseteq H$ for all $g\in S$; \\
 	 (b) $ghH= Hgh$ for all $g,h\in S $;\\
 	 (c) $HHg= gHH= HgH$ for all $g\in S$.\end{Def}
 
 \begin{Def} \cite{Sioson} A non-empty subset $I$ of a ternary semigroup $S$ is said to be an \emph{ideal} of $S$ if $SSI\subseteq I$, $SIS\subseteq I$ and $ISS\subseteq I$.
 \end{Def}

 \begin{Def}\cite{Km} An equivalence relation $\rho$ on a ternary semigroup $S$ is called \\
 	(i) \textit{left congruence} if $a\rho b\Rightarrow (sta)\rho (stb)$ for all $a,b,s,t\in S$;\\
 	(ii) \textit{right congruence} if $a\rho b\Rightarrow (ast)\rho (bst)$ for all $a, b, s, t \in S$;\\
 	(iii) \textit{lateral congruence} if $a\rho b\Rightarrow (sat)\rho (sbt)$ for all $a, b, s, t\in S $;\\
 	(iv) \textit{congruence} if it is left congruence, right congruence and lateral congruence.
 	\end{Def}
 
 \begin{Th}\cite{Km}\label{T:th1} An equivalence relation $\rho$ on a ternary semigroup $S$ is a congruence iff $(a,b), (c,d), (e,f)\in \rho\Rightarrow(ace,bdf)\in \rho ,\,\, \forall a,b,c, d, e, f\in S$. \end{Th}
 
\begin{Def} \cite{Willard} A Hausdorff topological space $X$ is called  a \emph{$k$-space} if it has the weak topology determined by the family of its compact subspaces.\end{Def}

\begin{Def} \cite{Willard}
A locally compact space  is  \emph{$\sigma$-compact}\index{$\sigma$-compact} if it can be expressed as the union of at most countable number of compact spaces.
\end{Def}

\begin{Def} \cite{Willard}
A topological space $X$ is called \emph{$k_{w}$-space} if it is the the union of a countable collection $\{K_{n}\}$ of compact subsets of $X$ so that a subset $A\subseteq X$ is closed whenever $A\cap K_{n}$ is close in $K_{n}$ for all $n\in \N$.\end{Def}

 So it follows that every $\sigma$-compact space is $k_w$-space.

\begin{Def}
If $\rho$ is a congruence (equivalence relation) on a topological ternary semigroup $S$, then it is called a \emph{closed congruence (closed equivalence relation)} if $\rho$  is a closed subset of $S\times S$.
If $\rho$ is an equivalence relation on a set $X$, then the set  $x\rho=\{y\in X: (x,y)\in \rho \}$ is called the \emph{$\rho$-class} of $X$ containing $x$. It is well known result that $\rho$-classes of $X$ are disjoint and $X$ is the union of all disjoint $\rho$-classes of $X$. The set $X/\rho$ of $\rho$-classes of $X$ is called \textit{quotient of $X$ mod $\rho$}. The function $\pi_\rho : X\rightarrow X /\rho$ which assigns to each $x\in S$, the $\rho$-class containing $x$ is called the \textit{natural map}. We note that for each $x\in X$, the set $\pi_\rho^{-1}(\pi_\rho(x))$  is the $\rho$-class of $X$ containing $ x $. This map $\pi_\rho$ is surjective.	
\end{Def}

\begin{Th}\cite{Km}\label{D:def} 
 Let $\rho$ be a congruence on a ternary semigroup $S$. Define a ternary multiplication on the quotient set $S/\rho$ by $(a\rho)(b\rho)(c\rho)=(abc)\rho,\,\ \forall a, b, c\in S$. Then $S/\rho$ is a ternary semigroup.	
\end{Th}

\begin{Th}\label{T:th4}
Let $S$ be a ternary semigroup, $I$  be an ideal of $S$ and $\rho_1:=(I\times I)\cup\Delta(S)$, where $\Delta(S):=\{(x,x):x\in S\}$  is the diagonal of $S \times S$. Then $\rho_1$ is a congruence on $S$.\end{Th}

It follows from Theorem \ref{D:def} that $S/\rho_1$ is a ternary semigroup.

\begin{Def}
Let $S$ be a ternary semigroup and $I$ be an ideal of $S$. Then the ternary semigroup $S/((I\times I)\cup \Delta(S))$ i,e $S/\rho_1$ is called \emph{Rees quotient ternary semigroup} of $S$ mod $I$ and denoted by $S/I$.\end{Def}

\begin{Def} \cite{Rrr}
A ternary semigroup $S$ is said to be a \emph{topological ternary semigroup} if there exists a \textbf{Hausdorff} topology on $S$ such that the ternary multiplication
$\left.\begin{array}{c}
S\times S\times S\longrightarrow S\\ (x, y, z)\longmapsto xyz
\end{array}\right\}$
is continuous, $S\times S\times S$ being equipped with the product topology.\label{Def: ttsg}
\end{Def}

\begin{Def}\cite{SSS1}
A ternary group $S$ is said to be a \emph{topological ternary group} if there exists a \textbf{Hausdorff} topology on $S$ such that the ternary multiplication on $S$ and the inversion map $x\mapsto x^{-1}\ (x\in S)$  both are continuous.

That the inversion map on $S$ is continuous is equivalent to the condition : for each $x\in S$ and each open nbd. $W$ of $x^{-1}$ in $S$, $\exists$ an open nbd. $V$ of $x$ in $S$ such that $V^{-1}\subseteq W$.
\end{Def}

 \section{Quotient of a  Topological Ternary Semigroup }
 
As mentioned in the beginning, in this section we study quotient of a topological ternary semigroup using congruence suitably. Then we impose conditions to make the quotient space a topological ternary semigroup. Following lemmas are very important for continuation of our discussion.

\begin{Lem} \label{L:lem1}  Let $S$ be a ternary semigroup and $\rho$ be a congruence on $S$. Also let $f$ and $g$ be ternary multiplications on $S$ and $S/\rho$ respectively. Let $\pi_\rho : S\rightarrow S/\rho$ be a map defined by $\pi_\rho(a)= a\rho,\,\ \forall a\in S$. Then the following diagram commutes:
	
	\[
	\begin{tikzcd}[scale=20]
		S\times S\times S \arrow{r}{f} \arrow{d}{\pi_\rho\times \pi_\rho\times \pi_\rho}& S \arrow{d}{\pi_\rho} \\
		(S/\rho)\times (S/\rho)\times (S/\rho) \arrow{r}{g}& S/\rho
	\end{tikzcd}
	\]
\end{Lem}

\begin{proof}
 $(g\circ(\pi_\rho \times \pi_\rho \times \pi_\rho))(a,b,c)=g((\pi_\rho \times \pi_\rho \times \pi_\rho)(a,b,c))= g (\pi_\rho (a), \pi_\rho(b), \pi_\rho(c))=g(a\rho, b\rho,c\rho)=(a\rho)(b\rho)(c\rho)=(abc)\rho =\pi_\rho(abc)=\pi_\rho(f(a,b,c))=(\pi_\rho\circ f)(a,b,c)$, $\forall\, a,b,c\in S$. Hence we have $\pi_\rho\circ f=g\circ (\pi_\rho \times \pi_\rho \times \pi_\rho)$. Therefore the diagram commutes.
\end{proof}

\begin{Lem}
   Let $S$ be a topological space and $\rho$ be a congruence on $S$. Let $\pi_\rho: S \rightarrow S/\rho$ be the natural map. Then $\rho= (\pi_\rho \times \pi_\rho)^{-1}(\Delta(S/\rho))$
\label{L: lem2}\end{Lem}

\begin{proof}
Straightforward.
\end{proof}

\begin{Lem}\cite{Carruth}  Let $\rho$ be a closed equivalence relation on a topological space $X$ and $K$ be a compact subset of $X$. Then $\pi_\rho^{-1}\pi_\rho(K)$ is closed.
\label{L:lem3}\end{Lem}

 \begin{Lem}
 \cite{Carruth}  Let $X$ be a compact Hausdorff space and $\rho$ be a closed equivalence on $S$. Then:
 \\(a) $\pi_\rho: X\rightarrow X/\rho$ is closed;
 \\(b) if $V$ is any open subset of $X$ then $V_{\rho}= \{x\in X: \pi_\rho^{-1}\pi_\rho(x)\subset V\}$ is open and $\pi_\rho$-saturated;	
 \\(c) if $A$ is a closed $\pi_\rho$-saturated subset of $X$ and $W$ is an open subset of $X$ such that $A\subseteq W$ then there exists an open $\pi_\rho$-saturated subset $U$ and a closed $\pi_\rho$-saturated subset $K$ such that $A\subset U \subset K \subset W$; 
 \\(d) $X/\rho$ is Hausdorff.
 \label{L:lem4}\end{Lem}

As mentioned in the beginning of this section, we now consider the base $ S $ as a topological ternary semigroup, $\pi_\rho: S \rightarrow S/\rho$ as a quotient map and induce some algebraic and topological conditions to make the quotient space a topological ternary semigroup. First and most important result of this section is the following:

\begin{Th}
\label{A: theo} 
 Let $S$ be a topological ternary semigroup and $\rho$ be a closed congruence on $S$ such that $\pi_\rho \times \pi_\rho\times\pi_\rho: S\times S\times S\rightarrow (S/\rho)\times (S/\rho)\times(S/\rho)$ is a quotient map. Then the induced ternary multiplication on $S/\rho$  is continuous. Moreover, if $S$ and $S/\rho$ both are locally compact, then $S/\rho$ is a topological ternary semigroup under the induced ternary multiplication.
\end{Th}

 \begin{proof}
 Let $f$ and $g$ be corresponding ternary multiplications on $S$ and $S/\rho$ respectively. From Theorem \ref{D:def}, $S/\rho$ is a  ternary semigroup. Also we have shown earlier in Lemma \ref{L:lem1}  that the following diagram commutes:
 $$
 \begin{tikzcd}[scale=20]
 	S\times S\times S \arrow{r}{f} \arrow{d}{\pi_\rho\times \pi_\rho\times \pi_\rho}& S \arrow{d}{\pi_\rho} \\
 	(S/\rho)\times (S/\rho)\times (S/\rho) \arrow{r}{g}& S/\rho
 \end{tikzcd}
 $$
 i,e $g\circ(\pi_\rho \times\pi_\rho\times \pi_\rho)= \pi_\rho\circ f$. Since $\pi_\rho$ is a quotient map and $f$ is continuous, so $\pi_\rho\circ f$ is continuous. Now let $U$ be any open set in $S/\rho$. Then $(\pi_\rho\circ f)^{-1}(U)= (g\circ(\pi_\rho \times\pi_\rho\times \pi_\rho))^{-1}(U)=((\pi_\rho \times\pi_\rho\times \pi_\rho)^{-1}) g^{-1}(U)$ is open (since $\pi_\rho\circ f$ is continuous). Now $\pi_\rho \times\pi_\rho\times \pi_\rho$ being a quotient map we can say that $g^{-1}(U)$ is open and hence $g$ is continuous (since $U$ is open).
  
 Since $S$ and $S/\rho$ both are locally compact and $\pi_\rho: S\rightarrow S/\rho$ is a quotient map, we have that $\pi_\rho \times i_{S}: S \times S \rightarrow (S/\rho)\times  S$ and $i_{(S/\rho)} \times \pi_\rho: (S/\rho)\times S\rightarrow (S/\rho)\times (S/\rho)$   are quotient maps, where $i_{S}$ and $i_{(S/\rho)}$ are  identity mappings on $S$ and $S/\rho$ respectively. Hence $(i_{(S/\rho)} \times \pi_\rho)\circ(\pi_\rho \times i_{S})=\pi_\rho \times \pi_\rho$  is a quotient map, as composition of two quotient maps is a quotient map (by 4.1 and 4.3 of XII 
 \cite{dugunji}). Again by Lemma \ref{L: lem2},$\rho= (\pi_\rho \times\pi_\rho)^{-1}(\Delta(S/\rho))$. Now it is given that $\rho$ is a closed congruence. Since $\pi_\rho \times \pi_\rho$ is a quotient map, we have that $\Delta(S/\rho)$ is closed. Therefore $S/\rho$ is a Hausdorff space. Hence $S/\rho$ is a topological ternary semigroup.
\end{proof}

\begin{Cor}
 Let $S$ be a compact topological ternary semigroup and $\rho$ be a closed congruence on $S$. Then $S/\rho$ is a compact topological ternary semigroup. \label{T: theo}\end{Cor}

\begin{proof}   Since $S$ is compact, it is locally compact. Now $\pi_\rho: S\rightarrow S/\rho$ is an onto  quotient map and hence continuous surjection. So $S/\rho$ is compact, as continuous image of a compact space is compact. Again proceeding in the same way as we have done in last theorem,  we can show that $S/\rho$ is a Hausdorff space (as compactness of $S$ and $S/\rho$ imply their local compactness). Again $\pi_\rho: S\rightarrow S/\rho$ is a closed map, by Lemma \ref{L:lem4} (a). We know that  Cartesian product of closed sets is a closed set. Therefore we have that $\pi_\rho\times \pi_\rho\times \pi_\rho$ is a closed map from $S\times S\times S $ to  $(S/\rho)\times (S/\rho) \times (S/\rho)$. Now continuity of $\pi_\rho$ implies that $\pi_\rho \times\pi_\rho\times \pi_\rho$ is also continuous. Therefore $\pi_\rho \times\pi_\rho\times \pi_\rho$ is a quotient map (since it is continuous and closed map). Hence by Theorem \ref{A: theo},ternary multiplication on $S/\rho$ is continuous.  Consequently, $S/\rho$ is a compact topological ternary semigroup.
\end{proof}

Above corollary is also valid if we replace closed congruence $\rho$ by $(I\times I)\cup\Delta(S)$, where $I$ is a closed ideal of $S$. This is basically Rees version of the same problem.

\begin{Cor}
Let $S$ be a compact topological ternary semigroup and $\rho_1= (I\times I)\cup\Delta(S)$, where $I$ is a closed ideal of $S$ and $\Delta(S)$ is the diagonal of $S$.  Then $S/\rho_1$ is a compact topological ternary semigroup. \label{C: cor2}
\end{Cor}

\begin{proof}
 $\rho_1$ is a congruence on $S$ by Theorem \ref{T:th4}. Since $I$ is a closed ideal, we have $I\times I$ is closed. Also $\Delta(S)$ is closed since $S$ is a Hausdorff space. Therefore  $\rho_1$ is a closed congruence on $S$. Hence by Corollary \ref{T: theo}, $S/\rho_1$ is a compact topological ternary semigroup. 
\end{proof}

We now show that we can replace compact space by $k_w$-space but still having the same result.
\begin{Th}
\label{T: th3}Let $S$ be a topological ternary semigroup which is a $k_{w}$-space and $\rho$ be a closed congruence on $S$. Then $S/\rho$ is a topological ternary semigroup.\end{Th}

\begin{proof} Since $S$ is a topological ternary semigroup which is a $k_{w}$-space, $S$ is a Hausdorff $k_{w}$-space. As cartesian product of Hausdorff spaces is Hausdorff, $S\times S$ and $S\times S\times S$  are Hausdorff spaces. Since $S$  is $k_{w}$-space and $k_{w}$-spaces are finitely productive (Milnor in \cite{Milnor} proved it for two spaces and applying this several times finitely productive property can be proved), we have that $S\times S$ and $S\times S \times S$ are Hausdorff $k_{w}$-spaces.  In \cite{Morita}, Morita proved that any $k_{w}$-space is the quotient of a $\sigma$-compact locally compact Hausdorff space. Hence there exists a $\sigma$ compact locally compact Hausdorff space $S'$ such that $S$ is the quotient of $S'$. Now $\pi_\rho: S\rightarrow S/\rho$ is a quotient mapping. Since $S$ is a quotient space of $S'$, there exists a quotient map $p:S'\rightarrow S$. Now define $m: S'\rightarrow S/\rho$ such that the following diagram commutes:
	
	$$
	\begin{tikzcd}
		S' \arrow{r}{p} \arrow{dr}{m}  & S \arrow{d}{\pi_\rho} \\
		& S/\rho
	\end{tikzcd}
	$$
 Therefore from this diagram, $m=\pi_\rho\circ p$. Then $m$ is a quotient map as it is a composition of two quotient maps $\pi_\rho$ and $p$. Hence $S/\rho$ is a quotient space of $S'$. Now by Bourbaki, N. General Topology \cite{Bourbaki}, $S/\rho$ is a Hausdorff space. Again by Morita \cite{Morita} $S/\rho$ is a $k_{w}$-space. Since Hausdorff $k_{w}$-spaces are finitely productive,  the space $(S/\rho )\times (S/\rho) \times (S/\rho)$ is a Hausdorff $k_{w}$-space. Also $S$ is a Hausdorff $k_{w}$-space.
 
 We now show that the map $\pi_\rho\times \pi_\rho\times \pi_\rho: S\times S \times S\rightarrow (S/\rho)\times (S/\rho)\times (S/\rho)$ is a quotient map. In fact, $\pi_\rho: S\rightarrow S/\rho$ being a quotient map  we have that $\pi_\rho \times \pi_\rho:S\times S \rightarrow (S/\rho)\times (S/\rho)$ is a quotient map, as $S$ and $(S/\rho)\times (S/\rho)$ both are Hausdorff $k_{w}$-spaces (by Theorem 1.5 of \cite{Michael}). Now using the same argument on $\pi_\rho$ and $\pi_\rho\times \pi_\rho $, we have that $\pi_\rho \times \pi_\rho \times \pi_\rho$ is a quotient map, as $S$ and $(S/\rho)\times (S/\rho)\times (S/\rho)$  are Hausdorff $k_{w}$-spaces. Already we have proved in Theorem \ref{A: theo}  that induced ternary multiplication on $S/\rho$ is continuous when $\pi_\rho\times \pi_\rho\times \pi_\rho$ is quotient. Therefore $S/\rho$ is a topological ternary semigroup.
\end{proof}
\begin{Rem}
	It is to be noted that Corollary \ref{T: theo} and Corollary \ref{C: cor2}  which we have proved earlier are direct consequences of  the Theorem \ref{T: th3}, since every compact topological space is a $k_{w}$-space.
\end{Rem}
\begin{Cor}
	Let $S$ be a locally compact $\sigma$-compact topological ternary semigroup and  $\rho$ be a closed congruence on $S$. Then $S/\rho$ is a topological ternary semigroup. \label{ch6:t:locally comp1}
\end{Cor}
\begin{proof} We know that every Hausdorff locally compact $\sigma$-compact topological space is a $k_{w}$-space. Since $S$ is a topological ternary semigroup so it is Hausorff space according to  definition \ref{Def: ttsg}. Hence by Theorem \ref{T: th3}, $S$ is a topological ternary semigroup. 
\end{proof}

\begin{Cor}
	Let $S$ be a locally compact $\sigma$-compact topological ternary semigroup and let $I$ be a closed ideal of $S$. Then $S/(I\times I)\cup \Delta(S)$ is a topological ternary semigroup.
\end{Cor}

\begin{proof} We find that $(I\times I)\cup \Delta(S)$ is a closed congruence on $S$. Therefore by  Corollary \ref{ch6:t:locally comp1}, $S/(I\times I)\cup \Delta(S)$ is a topological ternary semigroup.
\end{proof}
We now provide an alternative proof of the Corollary \ref{ch6:t:locally comp1} independently that is with out  help  of the Theorem \ref{T: th3}.  To proceed we need the following lemmas.

\begin{Lem} \cite {Carruth}\label {L:lem5}  Let $X$ be a locally compact $\sigma$-compact space. Then there exists a sequence $\{K_{n}\}$ of compact subsets of $X$  such that $X=\cup\{K_{n}\}_{n\in \N}$ and $K_{n}\subset K_{n+1}^{o}$ (the interior of $K_{n+1}$ in $X$) for each $n\in \N$. \end{Lem}

\begin{Lem} \label {L:lem6}  Let $X,\,\ Y,\,\ Z$ and $W$ be topological spaces and $A,B,C$ be compact subsets of $X,Y,Z$ respectively. Also assume that $f: X\times Y \times Z\rightarrow W$ be a continuous function and  $D$ be an open subset of $W$ containing $f(A\times B\times C)$. Then there exist open sets $P, Q$ and $ R$ in $X,\,\ Y $ and $Z$ respectively such that $A\subseteq P,\,\ B\subseteq Q,\,\ C\subseteq R$ and $f(P\times Q\times R)\subseteq D$.\end{Lem}

\begin{proof}  Since $f$ is continuous, $f^{-1}(D)$ is an open subset of $X\times Y \times Z$ containing $A\times B\times C$. Then for every $(a,b,c)\in A\times B\times C$, there exist $L, M$ and $N$ open in $X, Y$ and $Z$ respectively such that $(a,b,c)\in L\times M\times N \subseteq f^{-1}(D)$. Since $B$ is compact, for  fixed $a\in A$ and $c\in C$ there exist open sets $L_{1}, L_{2}, L_{3},......,L_{k}$ containing $a$; $ N_{1},N_{2},.....,N_{k}$ containing $c$ and corresponding open sets $M_{1}, M_{2},.....,M_{k}$ in $Y$ such that $B\subseteq F:=\displaystyle\bigcup_{i=1}^{k} {M_{i}}$. Let $ E:=\displaystyle\bigcap_{i=1}^{k} L_{i}$, $G:= \displaystyle\bigcap_{i=1}^{k} N_{i}$. Then $E, F$ and $G$ are open in $X, Y$ and $Z$ respectively. Also $a\in E$, $B\subseteq F$, $c\in G$ and $E\times F\times G\subseteq f^{-1}(D)$. 
	
Now $A$ being compact, for each fixed $ c\in C $ there exist open sets $E_{1},E_{2},....,E_{m}$ in $X$; $F_{1},F_{2},......,F_{m}$ in $Y$ and $G_{1},G_{2},......,G_{m}$ in $Z$ such that	$A\subseteq H:=\displaystyle\bigcup_{i=1}^{m}E_{i}$,$B\subseteq I:=\displaystyle\bigcap_{i=1}^{m} F_{i}$, $c\in J:=\displaystyle\bigcap_{i=1}^{m}G_{i}$. Clearly, $ H,I,J $ are open in $ X,Y,Z $ respectively such that $H\times I\times J\subseteq f^{-1}(D)$. 
	
Similarly $C$ being compact, there exist open sets $H_{1},H_{2},....,H_{n}$ in $X$; $I_{1},I_{2},......,I_{n}$ in $Y$ and $J_{1},J_{2},......,J_{n}$ in $Z$ such that $A\subseteq P:=\displaystyle\bigcap_{i=1}^{m}H_{i}$, $B\subseteq Q:=\displaystyle\bigcap_{i=1}^{m} I_{i}$, $C\subseteq R:=\displaystyle\bigcup_{i=1}^{m}J_{i}$. Then $ P,Q,R $ are open in $ X,Y,Z $ respectively such that $P\times Q\times R\subseteq f^{-1}(D)$ $ \implies f(P\times Q\times R)\subseteq ff^{-1}(D)\subseteq D $.
\end{proof}

 So we are in a position to prove the Corollary \ref{ch6:t:locally comp1} which is given below as a theorem.

\begin{Th}
Let $S$ be a locally compact $\sigma$-compact topological ternary semigroup and  $\rho$ be a closed congruence on $S$. Then $S/\rho$ is a topological ternary semigroup. \label{ch6:t:locally comp}\end{Th}

\begin{proof}
From the proof of Theorem 1.56 of  \cite {Carruth}, we have that $\pi_\rho\times \pi_\rho$ is a quotient map, where $\pi_\rho : S\rightarrow S/\rho$ is a quotient map. Using same proof it is easy to prove that $S/\rho$ is a Hausdorff space. Now we show that $\pi_\rho \times \pi_\rho\times \pi_\rho$ is a quotient map.  Then we  apply Theorem  \ref{A: theo} to prove the theorem. 	

Let $Q$ be a subset of the product space  $(S/\rho)\times (S/\rho)\times (S/\rho)$ such that $(\pi_\rho\times \pi_\rho\times \pi_\rho)^{-1}(Q)$ is  an open set in $S\times S\times S$. We show that $Q$ is an open set in $(S/\rho)\times (S/\rho)\times (S/\rho)$. Let for $a,b, c\in S$, $(\pi_\rho(a),\pi_\rho(b),\pi_\rho(c))\in Q$. Then $(\pi_\rho^{-1}\pi_\rho(a)\times \pi_\rho^{-1}\pi_\rho(b)\times \pi_\rho^{-1}\pi_\rho(c))\subseteq (\pi_\rho\times \pi_\rho \times \pi_\rho)^{-1}(Q)$.  By Lemma \ref {L:lem5},  there exists a sequence $\{K_{n}\}_{n\in \N}$ of compact subsets of $S$ such that $S=\bigcup \{K_{n}\}_{n\in \N}$ and $K_{n}\subseteq K_{n+1}^{0}$ for each $n\in \N$.  Assume without any loss of generality that $a, b$ and $c$ lie in $K_{1}$.
Let $U_{0}= \pi_\rho^{-1}\pi_\rho(a)\bigcap K_{1}, V_{0}= \pi_\rho^{-1}\pi_\rho(b)\bigcap K_{1}$ and $W_{0}= \pi_\rho^{-1}\pi_\rho(c)\bigcap K_{1}$. Then $ U_0,V_0,W_0 $ being $ \pi_{\rho} $-saturated and closed subsets of the compact set $ K_1 $ are compact. So by Lemma \ref{L:lem6}, there exist sets $G$, $H$ and $I$ which are open in $K_{1}$ such that $ U_{0}\subseteq G, V_{0} \subseteq H, W_{0}\subseteq I $ and $U_{0}\times V_{0} \times W_{0}\subseteq G \times H \times I\subseteq (\pi_\rho \times \pi_\rho \times \pi_\rho)^{-1}(Q)\bigcap(K_{1}\times K_{1}\times K_{1})$.
 Using Lemma \ref{L:lem4}(c), and the fact that $K_{1}$ is compact, we obtain sets $U_{1}, V_{1}$ and $W_{1}$ which are closed in $K_{1}$ (and hence compact) and $\pi_\rho|_{K_{1}}-$saturated such that $U_{1}$ is $K_{1}$-nbd. of $U_{0}$, $V_{1}$ is $K_{1}$-nbd. of $V_{0}$, $W_{1}$ is $K_{1}$-nbd. of $W_{0}$; $U _{1}\subseteq G, V_{1}\subseteq H, W_{1}\subseteq I$. Also by Lemma \ref{L:lem3}, $\pi_\rho^{-1}\pi_\rho(U_{1}),\pi_\rho^{-1}\pi_\rho(V_{1}),\pi_\rho^{-1}\pi_\rho(W_{1})$ are closed. Again by Lemma  \ref {L:lem6}
there exist $L$, $M$ and $N$ which are open in $K_{2}$ so that $(\pi_\rho^{-1}\pi_\rho(U_{1})\cap K_{2}) \times (\pi_\rho^{-1}\pi_\rho(V_{1})\cap K_{2}) \times (\pi_\rho^{-1}\pi_\rho(W_{1})\cap K_{2})\subseteq L\times M \times N\subseteq (\pi_\rho \times \pi_\rho \times \pi_\rho)^{-1}(Q)\cap (K_{2}\times K_{2} \times K_{2})$. As done above, there are sets $U_{2}, V_{2}$ and $W_{2}$ which are closed in $K_{2}$ (and hence compact) and $\pi_\rho|_{K_{2}}$-saturated such that $U_{2}$ is a $K_{2}$-nbd. of $\pi_\rho^{-1}\pi_\rho(U_{1})\cap K_{2}$, $V_{2}$ is a $K_{2}$-nbd. of $\pi_\rho^{-1}\pi_\rho(V_{1})\bigcap K_{2}$ and $W_{2}$ is a $K_{2}$-nbd. of $\pi_\rho^{-1}\pi_\rho(U_{1})\cap K_{2}$, $U_{2}\subseteq L$, $V_{2}\subseteq M$ and $W_{2}\subseteq N$. In this way we find recursively three towered sequences $\{U_{n}\}$, $\{V_{n}\}$ and $\{W_{n}\}$  that are closed in $K_{n}$
(and hence compact) and $\pi_\rho|_{K_{n}}$-saturated such that $U_{n}$ is a $K_{n}$-nbd. of $\pi_\rho^{-1}\pi_\rho(U_{n-1})\cap K_{n}$, $V_{n}$ is a $K_{n}$-nbd. of $\pi_\rho^{-1}\pi_\rho(V_{n-1})\cap K_{n}$ and $W_{n}$ is a $K_{n}$-nbd. of $\pi_\rho^{-1}\pi_\rho(W_{n-1})\cap K_{n}$, for all $n\in \N$. Let  $U=\bigcup\{U_{n}\}_{n\in \N}$,$V=\bigcup\{V_{n}\}_{n\in \N}$ and $W=\bigcup\{W_{n}\}_{n\in \N}$. It is clear that $\pi_\rho^{-1}\pi_\rho(a)\subseteq U$, $\pi_\rho^{-1}\pi_\rho(b)\subseteq V$ and $\pi_\rho^{-1}\pi_\rho(c)\subseteq W$. Observe that $U$ is $\pi_\rho$-saturated, since $U=\bigcup\{\pi_\rho^{-1}\pi_\rho(U_{n})\}_{n\in \N}$. Similarly $V$ and $W$ are $\pi_\rho$-saturated. To see that $U$ is open, let $x\in U$. Then $x\in U_{n}$ for some $n\in \N$.
Now $U_{n+1}$ is a $K_{n+1}$-nbd. of $U_{n}$ and $U_{n}\subseteq K_{n+1}^{0}$. It follows that $U_{n+1}\cap K_{n+1}^{0}$ is a $K_{n+1}^{0}$-nbd. of $U_{n}$. Since $K_{n+1}^{0}$ is open in $S$, $U_{n+1}\cap K_{n+1}^{0}$ is a nbd. of $U_{n}$. Also we note that $x\in U_{n}\subseteq U_{n+1}\cap K_{n+1}^{0}\subseteq U_{n+1}\subseteq U$. Thus $U$ is a nbd. of $x$. Consequently, $U$ is open. Similarly we can show that  $V$ and $W$ are open. Now $(\pi_\rho(a),\pi_\rho(b),\pi_\rho(c))\in \pi_\rho(U)\times \pi_\rho(V)\times\pi_\rho(W)$ and $\pi_\rho(U)\times \pi_\rho(V)\times \pi_\rho(W)$
is open, since $U$, $V$ and $W$ are open $\pi_\rho$-saturated subsets of $S$. Therefore, $\pi_\rho\times \pi_\rho \times \pi_\rho$ is an open mapping and hence it is a quotient map. Therefore by Theorem \ref{A: theo}, $S/\rho$ is a topological ternary semigroup.
\end{proof}




 Before going to our next result, we need the following concepts.
 
 \begin{Def} \cite{Behnam} Let $X$ be a topological space and $\{Y_\alpha\}_{\alpha \in J}$ be a family of its subspace. If $X$ has a weak topology induced by $\{Y_\alpha\}_{\alpha \in J}$, then we denote this by $X=\Sigma_{\alpha\in J}Y_\alpha$. If $X$ has the weak topology induced by two subspaces $Y$ and $Z$ then we denote it by $X=Y\oplus Z$.
 \end{Def}
 
 \begin{Def}\cite{Behnam} A closed ideal $I$ of $S$ is called regular, if for any $s\in S\smallsetminus I$, there exist open nbds. $V_s$ and $W_I$ of $s$ and $I$ respectively, such that $V_s\cap W_I=\emptyset$.
 \end{Def}  
  
\begin{Th} \cite{Behnam}
Let $S$ be a topological semigroup and $I$ be a closed ideal of $S$. The following conditions on $S$  are equivalent:
\\(a) $S/I$ is a Hausdorff $k_w$-space;
\\(b) $I$ is regular in $S$ and $S=I\oplus(\sum_{n\in J}K_{n}')$, where $J\subseteq \N$ and either  $K_{n}'\cap I=\emptyset$ and $K_{n}'$ is compact in $S$ or $K_{n}'=I\cup K_{n}$, where $K_{n}$ is compact in the topological space $S_{I}:=(S\smallsetminus I^0,\tau)$, $ \tau $ being defined by 
$\tau := \{O: O \text{ is open in }S\text{ and either }O\cap I=\emptyset\text{ or }I\subseteq O\}$.\end{Th}

From the proof of above result we can comment that same is true in topological ternary semigroup. So we restate it for topological ternary semigroup without giving its  proof.

 \begin{Th}
 \label{T: th2} Let $S$ be a topological ternary semigroup and $I$ be a closed ideal of $S$. The following conditions on $S$  are equivalent:
 \\(a) $S/I$ is a Hausdorff $k_w$-space;
 \\(b) $I$ is regular  in $S$ and $S=I\oplus(\sum_{n\in J}K_{n}')$, where $J\subseteq \N$ and either  $K_{n}'\cap I=\emptyset$ and $K_{n}'$ is compact in $S$ or $K_{n}'=I\cup K_{n}$, where $K_{n}$ is compact in the topological space $S_{I}:=(S\smallsetminus I^0,\tau)$, $ \tau $ being defined by 
 $\tau := \{O: O \text{ is open in }S\text{ and either }O\cap I=\emptyset\text{ or }I\subseteq O\}$.
 \end{Th}

\begin{Th} Let $S$ be a topological ternary semigroup and $I$ be a regular closed ideal of $S$. If $S=I\oplus(\sum_{n\in Y}K_{n}')$ where $Y\subseteq \N$ and either $K_{n}'\cap I=\emptyset$ and $K_{n}'$ is compact in $S$ or $K_{n}'=I\cup K_{n}$ where $K_{n}$ is compact in $S_{I}$, then $S/I$ is a topological ternary semigroup.
\end{Th}

\begin{proof} According to Theorem \ref{T: th2},  $S/I$ is a Hausd\"{o}rff $k_w$-space.  Again from Theorem \ref{T: th3},  $(S/I)\times (S/I)\times (S/I)$ is a Hausdorff space and $\pi_I \times \pi_I \times \pi_I$ ( where $\pi_I : S\rightarrow S/I$ is a quotient map) is a quotient map. Hence by  Theorem \ref{A: theo}, the result follows.
\end{proof}

\begin{Th} Let $S$ be a topological ternary semigroup which is a $k$-space and $I$ be a regular closed ideal of $S$. If $S\smallsetminus I^0$ is a $K_{w}$-space, then $S / I$ is a topological ternary semigroup.\end{Th}

\begin{proof} 
Let $S$ be a topological ternary semigroup and $I$ be a regular closed ideal of $S$ with $S\smallsetminus I^0$ is a $k_w$-space.  Then there exist countable number of compact sets, say, $\{L_n\}_{n\in P}$ where $P\subseteq \N$ such that $S\smallsetminus I^0=\sum_{n\in P}L_{n}$. Define \[\begin{array}{l}{L_{n}' :=\left\{\begin{array}{l} {L_{n}\;,\;\text{if} \,\ L_{n}\cap I= \emptyset} \\ {L_{n}\cup I\;,\; \text{otherwise}} \end{array}\right.} \end{array}\]\\Then we see that $S=I\oplus(\sum_{n\in P}L_{n}')$ and $L_{n}'$ satisfies the  condition of the Theorem \ref{T: th2} and hence $S/I$ is a $k_w$-space. Since $S$ and $(S/I) \times (S/I) \times (S/I)$ are Hausdorff $k$-spaces, by Theorem \ref{T: th3}, we have $\pi_I \times \pi_I \times \pi_I$ is a quotient map (where $\pi_I : S\rightarrow S/I$ is the quotient map). Hence by \ref{A: theo}, $S/I$ is a topological ternary semigroup.	
\end{proof}

There is another way of introducing quotient structure on a particular type of topological ternary semigroup. This particular topological ternary semigroup is topological ternary group. We shall make  quotient of it with the help of a normal  ternary subgroup.

Let $S$ be a topological ternary group and $H$ be a normal ternary subgroup of $S$. Then from Theorem 3.3.1 of \cite{Sj}, we know that the set $S/H: = \{xHH : x\in S\}$ of all the left cosets of $S$ becomes a ternary group under the ternary product $(xHH, yHH, zHH)\mapsto (xyz)HH$ \big[note that $(x,y,z)\mapsto xyz$ is the ternary multiplication from $S\times S\times S$  to $S$\big]. Define a mapping $\pi_H : S\rightarrow S/H$ by $\pi_H(x)=xHH ,\,\ \forall x\in S$. We now introduce a topology on $S/H$ in the following manner: a subset $\mathcal U$ of $S/H$ is open in $S/H$ if and only if $\pi_H^{-1}(\mathcal U)$ is open in $S$. Then $\pi_H$ induces a topology on $S/H$ which is the quotient topology and the map $\pi_H$ is a quotient map. We now prove the following theorem which ensures that $S/H$ with the quotient topology and ternary multiplication as mentioned above, becomes a topological ternary group. Following results are essential to prove our main theorem.

\begin{Th} \cite{SSS1}
	Let $S$ be a topological ternary group. Then for any $a,b\in S$ the following four maps are homeomorphisms from $S$ onto itself.\\
	{\em(i)} The left translation $l_{ab}:x\longmapsto abx\ (x\in S)$;\\
	{\em(ii)} The right translation $r_{ab}:x\longmapsto xab\ (x\in S)$;\\
	{\em(iii)} The lateral translation $t_{ab}:x\longmapsto axb\ (x\in S)$;\\
	{\em(iv)} The inversion $i_{_S}:x\longmapsto x^{-1}\ (x\in S)$
	\label{ch2:t:tg_homeo}.\end{Th}

From above Theorem \ref{ch2:t:tg_homeo} we can say that for any open set $U$ in a topological ternary group $S$ and for any $a,b\in S$, the sets $abU,Uab,aUb,U^{-1}$ are open in $S$ as well, since $l_{ab},r_{ab},t_{ab},i_{_S}$ are homeomorphisms from $S$ onto itself.

An immediate corollary of Theorem \ref{ch2:t:tg_homeo} is the following:
\begin{Cor}
	Let $S$ be a topological ternary group, $U$ be open in $S$ and $A$ be any set in $S$. Then 
	$AAU, AUA, UAA$ are open in $S$.
	\label{ch2:c:closurecor}\end{Cor}

\begin{proof} 
	 We have $AAU =\bigcup \{abU: a,b\in A\}$ and given that $U$  is open.  Therefore as an immediate consequence of Theorem \ref{ch2:t:tg_homeo}, $abU$  is open. Since arbitrary union of open sets is open, it follows that $AAU$ is open. Similarly other results follow.
\end{proof}

\begin{Th} Let $S$ be a topological ternary group and $H$ be a ternary subgroup of $S$. If $H$ is open in $S$ then $H$ is closed in $S$.\label{ch2:t:clopen}\end{Th}

\begin{proof} 
	
If $H$ is open then $\forall s\in S$, $sHH$ is also open.  To justify it, we have that $sHH=\displaystyle\bigcup_{h\in H}shH$. Now each $shH$ is open [since $H$ is a ternary sugroup of $S$ and $S$ is a topological ternary group]. So $\forall s\in S$, $sHH=\displaystyle\bigcup_{h\in H}shH$ is open, as arbitrary union of open sets is open. 

Again $\forall h\in H$, $hHH=\displaystyle\bigcup_{h'\in H}hh'H=H$ [since $H$ is a ternary subgroup of $S$]. We now show that $\forall a\in S\smallsetminus H$, $aHH\cap H=\emptyset$ and $S\smallsetminus H=\displaystyle\bigcup_{a\in S\smallsetminus H}aHH$. If possible let $x\in aHH\cap H$.  Then $x=ah_1h_2=h_3$ for some $h_1,h_2,h_3\in H$. So $a= h_3h_2^{-1}h_1^{-1}\in H$, a contradiction to the fact that $a\in S\smallsetminus H$. Therefore  $\forall a\in S\smallsetminus H$, $aHH\cap H=\emptyset$. 

We now define a relation `$\rho $' on $S$ by $a\rho b$ iff $aHH=bHH$ for $a,b\in S$. Then it is easy to verify that `$\rho $' is an equivalence relation on $S$. Therefore it divides $S$ into disjoint partitions. Let $a\in S$. Then $cl(a)=\{b\in S: aHH=bHH\}=\{b\in S: ah_1h_2=bh_3h_4, h_i\,(i=1,\ldots,4)\in H\}=\{b\in S: b= ah_1h_2h_4^{-1}h_3^{-1}\in aHH\}=aHH$ [it is easy to prove that $b\in aHH \Leftrightarrow aHH=bHH$]. Therefore $S\smallsetminus H=\displaystyle\bigcup_{a\in S\smallsetminus H}aHH $ [since $H$ is also a class disjoint from each $aHH$, $a\in S\smallsetminus H$]. In the beginning already we have proved that $aHH$ is open for $a\in S$. Therefore $\displaystyle\bigcup_{a\in S\smallsetminus H}aHH$  is open. So $S\smallsetminus  H=\displaystyle\bigcup_{a\in S\smallsetminus H}aHH$ is open. Hence $H$ is closed. \end{proof}  
 Now we prove our main theorem.
\begin{Th}
	Let $S$ be a topological ternary group and $H$ be a closed normal ternary subgroup of $S$. Then $S/H$ is a topological ternary group.\end{Th}

\begin{proof} 
	$S/H$ is a ternary group by Theorem 3.3.1 of \cite{Sj}.
	Now we prove the theorem in three different steps. \\
	\textbf{\underline{Step-I} :}
	First we show that the ternary multiplication  on $S/H$ as defined above is continuous. Let $\pi_H: S\rightarrow S/H$ defined by $ \pi_H(x)=xHH,\forall\,x\in S $ be the quotient map. Let $U$ be an open subset of $S$. Then  $\pi_H(U)=\{xHH:x\in U\}$ which is open in $ S/H $, since $\pi^{-1}_H(\{xHH:x\in U\})=UHH$ which is an open set in $ S $ [by Corollary \ref{ch2:c:closurecor}]. Hence $\pi_H$ is an open map. It then follows that $\pi_H \times \pi_H \times \pi_H$ is also an open map, which is  continuous as well, since $\pi_H$ is continuous. Thus $\pi_H\times \pi_H \times \pi_H$ is a quotient map. Therefore by Theorem  \ref{A: theo},  we have that the ternary multiplication on $S/H$ is continuous.\\
	\textbf{\underline{Step-II} :} In this step we show that the inversion map is continuous on $S/H$. Let $f: S/H\rightarrow S/H$ be defined by $f(aHH)=a^{-1}HH, \forall\,a\in S$. Then $ f $ is the inversion map, as inverse of $ aHH $
	is $ a^{-1} HH $ \big[since $(aHH)(a^{-1}HH)(gHH)=aa^{-1}{g}HH=gHH$,  $\forall\,g\in G$ and  similarly $(gHH)(aHH)(a^{-1}HH)=(a^{-1}HH) (aHH)(gHH)=(gHH)(a^{-1}HH) (aHH)=gHH$, $\forall g\in S$\big].
	
   Let $\mathcal W$ be any open nbd. of $a^{-1}HH$ in $ S/H $. Then $W:=\pi^{-1}_H(\mathcal W)$ is an open set in $ S $ that contains  $a^{-1}$. Since $S$ is a topological ternary group, for this nbd. $W$ of $a^{-1}$, there exists a nbd. $V$ of $a$ in $ S $ such that $V^{-1}\subseteq W$. Let $ \mathcal V:=\pi_H(V) $. Since $ \pi_H $ is an open map (as shown in Step-I), $ \mathcal V $ is an open nbd. of $ aHH $ in $ S/H $. We claim that $ f(\mathcal V)\subseteq \mathcal W $. For this let $ xHH\in\mathcal V $, where $ x\in V $. Then $ x^{-1}\in W $ $ \implies $ $ \pi_H(x^{-1})\in\pi_H(W)=\mathcal W $ $ \implies $ $ f(xHH)=x^{-1}HH=\pi_H(x^{-1})\in\mathcal W $. Thus $ f(\mathcal V)\subseteq \mathcal W $. This shows that $ f $ is continuous at $ aHH $. Then arbitrariness of $ aHH $ proves that $ f $ is continuous on $ S/H $.\\	
	\textbf{\underline{Step-III} :} Finally, we show that $S/H$ is a Hausdorff space. In the proof of the Theorem \ref{ch2:t:clopen} we have shown that $ \{xHH:x\in S\} $ is the collection of all disjoint equivalence classes with respect to the equivalence relation $ \rho:=\{(x,y)\in S\times S:xHH=yHH\} $. Thus $ S/H=S/\rho $. 
	
	We claim that $\rho$ is closed. It follows from the fact that $ H $ is closed, ternary multiplication in $ S $ is continuous and $ S $ is Hausdorff. Again in Step-I we have proved that $ \pi_H $ is an open map. Then by a standard result of general topology (see \cite{Kelley})  we can say that $ S/\rho $ is Hausdorff and hence $ S/H $ is Hausdorff.
	
	In view of Step-I, II and III it follows that $S/H$ is a topological ternary group.
\end{proof}

{\bf Declarations :} 

{\bf Ethical Approval :} Not applicable 



{\bf Funding :} There is no funding for this research work. 

{\bf Availability of data and materials :} There is no associated datasets and materials for this research work.

\end{document}